\documentclass[12pt]{amsart}

\usepackage{graphicx}
\usepackage{amsthm}

\def\G{\Gamma}
\def\g{\gamma}
\def\gh{\hat{\gamma}}
\def\gt{\tilde{\gamma}}

\def\D{{\mathcal D}}
\def\H{{\mathcal H}}

\def\Hfty{\partial_{\infty}\H}
\DeclareMathOperator\diam{diam}
\DeclareMathOperator\dist{dist}

\def\N{{\mathbb N}}
\def\Z{{\mathbb Z}}
\def\eps{\varepsilon}
\def\ph{\varphi}

\theoremstyle{definition}
\newtheorem{dfn}{Definition}

\theoremstyle{plain}
\newtheorem{thm}{Theorem}
\newtheorem{lemma}{Lemma}
\newtheorem{cor}{Corollary}

\theoremstyle{remark}
\newtheorem{case}{Case}
\newtheorem*{rmk*}{Remark}

\title{Morse coding for a Fuchsian group of a finite covolume}
\author[Arseny Egorov]{Arseny Egorov,\\\textit{The Pennsylvania State University}}

\begin{document}

\begin{abstract}
We consider a Fuchsian group $\G$ and 
the factor surface $\H/\G$, which has constant curvature $-1$ and maybe a few singularities. If we lift the surface continuously to $\H$ (except for a subset of a lower dimention), we obtain a fundamental domain $\D$ of $\G$. This can be done in different ways, but we restrict the choice 
to 
Dirichlet domains, which always are convex polygonal subsets of $\H$. Given a generic geodesic on $\H$, one can produce a so called geometric Morse code (or the cutting sequence) of the geodesic with respect to $\D$. We prove that the set of Morse codes of all generic geodesics on $\H$ with respect to $\D$ forms a topological Markov chain, if and only if $\D$ is an ideal polygon.
\end{abstract}

\maketitle

\section{Introduction}

There have been several approaches to coding of geodesics on the hyperbolic plane with respect to fundamental domains of Fuchsian groups. One approach was introduced by Morse \cite{Morse}, in which the plane is tesselated according to the chosen domain and the action of the group, the sides of the tesselation are labeled in a certain way, and the geodesics are then coded using their intersections with the sides of the tesselation. Another approach was developed by Artin \cite{Artin} and for coding it uses the endpoints of geodesics on the absolute, which are identified with real numbers and are expanded into continued fractions.

Artin's approach was broadened to different types of continued fractions, e.g. \cite{Katok2}, \cite{MayerStromberg}, and different Fuchsian groups (the original one was used with the modular group), \cite{BowenSeries}, \cite{Series1, Series2}, but it requires some reduction process (i.e. finding an appropriate geodesic among $\G$-equivalent ones) before actual encoding. Moreover, the coding process may be defined differently for the same tessellation of the plane, which corresponds to different types of continued fractions. Being more synthetic than Morse's, this approach produces a set of sequences of integers, which is a topological Markov chain and is much easier to structurize than just a ``random'' set of sequences.

Morse's approach utilizes only the way the hyperbolic plane is tiled by fundamental domains of the group, which can be reproduced ``internally'' on the factor surface $\H/\G$ by marking curves on the surface along which it needs to be ``cut'' to obtain one of the fundamental domains on the plane. And Morse's approach is defined uniquely once a tessellation is chosen. 
So it is somewhat more natural and canonic than Artin's. We became interested, under what conditions Morse's coding produces a set of sequences, which is a topological Markov chain too. In \cite{KatokUgarcovici} the case of the modular group $PSL(2, \Z)$ is considered along with a slightly modified version of the coding, which uses integer numbers instead of the group's elements. Their result implies the results of this paper for the case of the modular group and its standard fundamental domain, and is the primary motivation for this work.

\subsection{Fuchsian groups and Dirichlet domains}

The following definitions provide us with basic terms:

\begin{dfn}
A group of orientation preserving isometries of the hyperbolic plane $\H$, acting discretely, is called a Fuchsian group.
\end{dfn}

Given a Fuchsian group $\G$, one can consider the factor space $\H/\G$. Since $\G$ acts discretely on $\H$, the space is an orbifold.

\begin{dfn}
If the volume of the factor surface $\H/\G$ is finite, the group is said to have a finite covolume.
\end{dfn}

\begin{dfn}
Let $\G$ be a Fuchsian group, and let $x \in \H$ be a point which is not fixed by any element of $\G$ except the identity. Then the Dirichlet domain of $\G$ with respect to $x$ is the set $\D_x(\G) = \{y \in \H: \dist(y, x) \leq \dist (y, gx)\ \forall g \in \G\}$.
\end{dfn}

A Dirichlet domain of a Fuchsian group is a fundamental domain of the group, \cite[\S 3.1]{Katok1}, and is a convex polygonal set on the union of the hyperbolic plane and its boundary at infinity. Siegel \cite[Theorem 4.1.1]{Katok1} tells us, that if the group has a finite covolume, the number of edges of a Dirichlet domain is finite. Clearly, a Dirichlet domain of a group with a finite covolume has no edges on $\Hfty$. In terms of the factor surface $\H/\G$, it may have a finite number of singular points and cusps, but cannot have funnels.

\subsection*{Some agreements}

\begin{itemize}
\item Given a Fuchsian group $\G$ and its Dirichlet domain $\D$, we will refer to the edges and vertices of the tessellation $\G\D$ as just ``edges'' and ``vertices'', because there are no edges or vertices considered in this paper other than those.

\item We will call the vertices on $\Hfty$ ``infinite'' and all the others --- ``finite''.

\item The geodesics are always directed. Then referring to the left and right hand side of the plane with respect to a given geodesic makes sense. If one geodesic crosses another one, we will say it is on the right from the other one, if its direction is from the left to the right, otherwise we will say it is on the left from the other one. Clearly, if one geodesic is on the right from another one, then the latter one is on the left from the former one and vice versa.

%
\end{itemize}

\subsection{Morse coding}

Consider a Fuchsian group $\G$ along with a Dirichlet domain $\D$. $\D$ shares its edges with some of its images (``copies'') under the action of $\G$. Let the images, sharing edges with $\D$, be $g_1\D$, $\dots$, $g_n\D$, where $g_1$, $\dots$, $g_n$ $\in \G$. Let us label the edges of $\D$ by the elements $g_1$, $\dots$, $g_n$, so that the edge shared with $g_i\D$ is labeled by $g_i$. This labeling is made inside the domain. Now we spread this labeling to the rest of the tessellation: for every $g \in \G$ the edge of $g\D$, corresponding to the one labeled by $g_i$ inside $\D$, is labeled by $g_i$ inside $g\D$.

\begin{figure}[ht]
	\includegraphics[width=\textwidth]{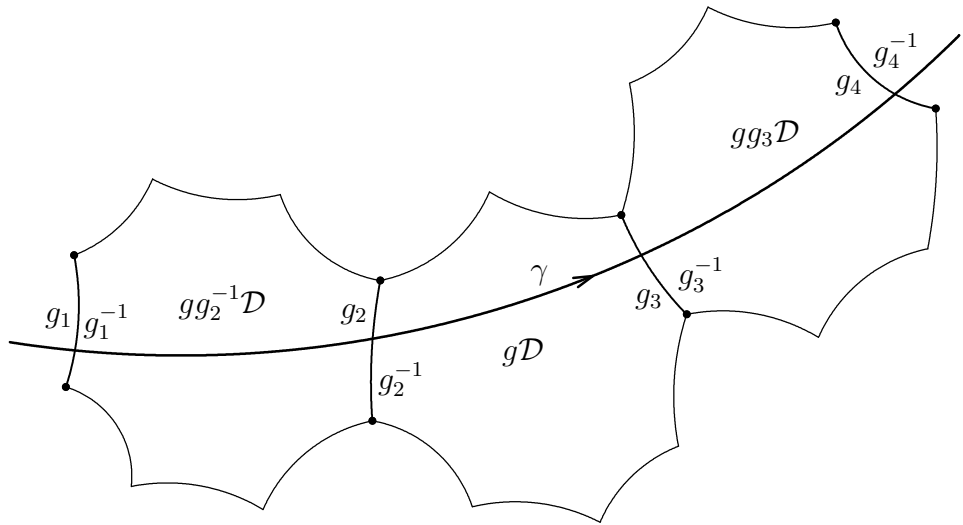}
	\setlength{\unitlength}{\textwidth}
	\begin{picture}(0,0)
		\put(-.46,.25){$g_1$}

		\put(-.42,.24){$g_1^{-1}$}
		\put(-.32,.26){$gg_2^{-1}\D$}
		\put(-.15,.255){$g_2$}

		\put(-.11,.18){$g_2^{-1}$}
		\put(.05,.295){$\g$}
		\put(.02,.21){$g\D$}
		\put(.16,.27){$g_3$}

		\put(.205,.29){$g_3^{-1}$}
		\put(.26,.44){$gg_3\D$}
		\put(.37,.495){$g_4$}

		\put(.41,.525){$g_4^{-1}$}
	\end{picture}
	\caption{A geodesic crossing multiple fundamental domains.}
\end{figure}

When the labeling is done, each edge on the plane is labeled on both sides. It is easy to show, that each edge is labeled by elements inverse to each other from opposite sides.

A generic geodesic on the hyperbolic plane avoids the vertices, finite or infinite. Given a generic geodesic, one can write down the biinfinite sequence of labels on the edges, which the geodesic crosses on its way (each time the geodesic crosses an edge, one writes down the label on the edge from the domain the geodesic is leaving). This sequence is called the geodesic's cutting sequence. It was used by Morse in his paper \cite{Morse}, and is also referred to as (geometric) Morse code of the geodesic. According to \cite{Series2}, it can be traced back to \cite{Koebe}.

It is easy to see from the geometric construction, that in the cutting sequence of a geodesic, an entry cannot be immediately followed by its inverse from $\G$. This fact follows from two observations: that every label is unique inside one domain, and that a geodesic cannot cross the same edge twice in a row.

\subsection{Main result}

We consider the set of cutting sequences of all generic geodesics on $\H$ with respect to a given Dirichlet domain $\D$.

\begin{dfn}
Let $\Sigma$ be a no more than countable alphabet. Let $\Lambda \subseteq \Sigma^\Z$ be a set of biinfinite sequences of elements of $\Sigma$. Assume there are $k \in \N$ and $\Lambda_k \subseteq \Sigma^{k+1}$, such that $\Lambda = \{\lambda \in \Sigma^\Z: (\lambda_n, \dots, \lambda_{n+k}) \in \Lambda_k, \forall\ n \in \Z\}$. Then $\Lambda$ is called a $k$-step topological Markov chain.

Let us say that $\Lambda$ is a topological Markov chain, if it is a $k$-step topological Markov chain for some $k \in \N$.

The elements of $\Lambda_k$ are called ``allowed'' $(k+1)$-tuples.
\end{dfn}

\begin{thm}
\label{thm:main}
Let $\G$ be a Fuchsian group with a finite covolume and a Dirichlet domain $\D$. Then the set of Morse codes of generic geodesics on $\H$ with respect to $\D$ is a topological Markov chain, if and only if $\D$ does not have finite vertices.
\end{thm}

\begin{rmk*}
The proof of sufficiency is essentially contained in \cite[\S1, pp.\ 603--604]{Series2}: one needs to consider the tree, which connects the elements of the $\G$-orbit of the point, defining $\D$, if and only if they are contained in copies of $\D$, sharing a side. Every infinite sequence consisting of $g_1$, $\dots$, $g_n$ will then correspond to a branch of the tree, starting at the root, and a point on $\Hfty$. Every biinfinite sequence will correspond to two points on $\Hfty$ and, thus, to a geodesic. One only needs to check this geodesic's cutting sequence coincides with the original biinfinite sequence.

Although this scheme produces a good proof, we present our own proof of sufficiency, which seems to be slightly shorter and doesn't need too many details to be checked.
\end{rmk*}

\section{Auxiliary statements}

Throughout this section $\G$ is a Fuchsian group of a finite covolume. $\D$ is a Dirichlet domain for $\G$.

\begin{lemma}[There are plenty of finite vertices on the plane]
\label{lemma:bdry}
Assume $\D$ has a finite vertex. Let $y \in \H$ be some point. Consider the set $\Xi$ of directions leading from $y$ to all the finite vertices on the plane. Then the set is dense in the set of all directions at $y$ (or, equivalently, between any two directions at $y$ there is a direction from $\Xi$).
\end{lemma}

\begin{proof}
There are two cases: $\D$ has no infinite vertices and otherwise.

\begin{case}Suppose $\D$ has no infinite vertices, that is $\D$ is bounded. Consider $\eta_1$ and $\eta_2$, two arbitrary directions at $y$. Let us prove there is a direction between $\eta_1$ and $\eta_2$ leading from $y$ to a finite vertex.

Let $\g_1$ and $\g_2$ be the two geodesics, passing through $y$ and following $\eta_1$ and $\eta_2$, respectively. Since the geodesics are not parallel, there should be a point $z$ inside the angle formed by the two geodesics, such that $\dist(z, \g_1\cup\g_2) > \diam(\D)$. Since $z$ is contained in a copy of $\D$, that copy is entirely contained inside the angle as well. Since $\D$ has a finite vertex, so does the copy, containing $z$. Let $V$ be a finite vertex of the copy. Clearly, $V$ is inside the angle between $\g_1$ and $\g_2$, too. Thus, the direction, leading from $y$ to $V$, is between $\eta_1$ and $\eta_2$.
\end{case}

\begin{case}Assume $\D$ has an infinite vertex. Since $\D$ has a finite vertex too, $\D$ should have a side with both a finite and an infinite vertex. Let us denote the infinite vertex of the side by $\xi'$ and the finite vertex by $X'$. By \cite[Theorem 4.2.5]{Katok1}, every infinite vertex is the fixed point of some parabolic element of $\G$, therefore so is $\xi'$.

$\G$ has a finite covolume, so by \cite[Lemma 4.5.3]{Katok1}, in the unit circle Poincare model the Euclidean diameters of sets $g_n\D$, $g_n \in \G$, go to $0$, as $n \to \infty$. So if one considers any point $\eta$ on $\Hfty$ and a circle $L$ of small radius $\eps > 0$ around it, there will always be an element $g_\eta(\eps) \in \G$, such that $g_\eta(\eps)\D$ is inside $L$, and thus its infinite vertex, corresponding to $\xi'$ is inside $L$, too. Since $\eta$ and $\eps$ are arbitrary, we showed, that infinite vertices corresponding to $\xi'$ in all copies of $\D$ are dense on $\Hfty$.

Consider a parabolic element $g' \in \G$, fixing $\xi'$. The points $g'^{m}X'$, $m \in \Z$, lie on an horocycle, passing through $\xi'$, and form a regular polygon with infinite number of sides. Each of these points is a finite vertex. As $m \to \pm\infty$, the geodesics, going from $y$ to $g'^mX'$, converge to the geodesic, going from $y$ to $\xi'$. So the direction from $y$ to $\xi'$ is an accumulation point of $\Xi$.

A similar argument can be used for any infinite vertex of the form $g\xi'$, $g \in \G$. Thus $\Xi$ has accumulation points at every direction from $y$ to $g\xi'$. As mentioned above, the set $\{g\xi': g\in\G\}$ is dense in $\Hfty$. But it is the same as saying that the directions from $y$ to $g\xi'$, are dense in the set of all directions at $y$. Hence $\Xi$ is dense in the set of all directions as well.
\qed 
\end{case}\renewcommand{\qed}{}
\end{proof}

\begin{lemma}
\label{lemma:prx}
If a geodesic passes near enough to a finite vertex, it intersects an edge, ending at the vertex. (There exists $\eps > 0$, depending only on $\D$, such that if a geodesic $\g$ passes at a distance less than $\eps$ from a finite vertex $X$, then $\g$ intersects one of the edges, ending at $X$.)
\end{lemma}

\begin{proof}
Since $\D$ has a finite volume, it has a finite number of edges. Let $\ph$ be the greatest angle of $\D$, and let $a$ be the length of the shortest edge of $\D$. Choose $\eps$ to be such, that a right triangle with the hypothenuse of length $a$ and an acute angle $\frac{\ph}{2}$ has the leg at that angle equal to $\eps$.

Suppose a geodesic $\g$ passes closer than $\eps$ to a finite vertex $X$, consider the perpendicular from $X$ to $\g$, let the base of it be point $y$. The minimal angle between $[X, y]$ and an edge ending at $X$ is at most $\frac{\ph}{2}$, the length of $[X, y]$ is less than $\eps$, and the length of the edge forming the minimal angle with $[X, y]$ is at least $a$, so the edge has to cross $\g$.
\end{proof}

\begin{lemma}
\label{lemma:n+1}
Suppose $\D$ has at least one finite vertex. Assume a geodesic $\g$ passes through a point $y$ and then crosses edges $\langle A_1, B_1\rangle$, $\dots$, $\langle A_n, B_n\rangle$. Then there are geodesics, $\g^l$ and $\g^r$, which pass through $y$, cross the same edges and at least one more after that, such that the new edge for $\g^l$ has a finite vertex on the right from $\g^l$, and the new edge for $\g^r$ has a finite vertex on the left from $\g^r$.
\end{lemma}

\begin{proof}
It is enough to prove the statement for $\g^r$. The proof for $\g^l$ is absolutely the same.

It is clear that we can rotate $\g$ around $y$ clockwise so that it still intersects all of $\langle A_i, B_i\rangle$. consider such a rotation of $\g$ and call the new geodsic $\gh^r$.

According to Lemma \ref{lemma:bdry}, there is a direction from $y$ between $\g$ and $\gh^r$, which leads to a finite vertex. Consider such a vertex $X$. Without loss of generality we may assume, that $X$ is located behind $\langle A_n, B_n\rangle$ from $y$ (since finite vertices constitute a discrete set on $\H$, there is only a finite number of finite vertices in the triangle bounded by $\g$, $\gh^r$, and $\langle A_n, B_n\rangle$, while there is an infinite number of finite vertices in the angle between $\g$ and $\gh^r$).

Now rotate $\gh^r$ around $y$ counterclockwise, so that the distance between $X$ and the new geodesic becomes less than $\eps$ (cf. Lemma \ref{lemma:prx}), but $X$ is still on the left. Call the new geodesic $\g^r$. Clearly, it passes through $y$ and crosses all of $\langle A_i, B_i\rangle$. Since it is closer than $\eps$ to $X$, it has to cross an edge ending at $X$, and the edge cannot be any of $\langle A_i, B_i\rangle$, because $X$ is behind all of them from $y$.
\end{proof}

\section{Proof of Theorem \ref{thm:main}}

To prove that a set $\Lambda \subset \Sigma^\Z$ is a topological Markov chain, we only need to present a number $k \in \N$ and a set of allowed $(k+1)$-tuples $\Lambda_k$.

To prove otherwise, we need to show that given any $k \in \N$ one can find a $(k+l+1)$-tuple $\lambda$, $l > 0$, such that no infinite sequence, containing $\lambda$, is in $\Lambda$, but every subsequence of $\lambda$ of length $k+1$ is contained in some infinite sequence from $\Lambda$.

\begin{proof}
In our case the alphabet consists of the labels put on the edges of $\D$: $\Sigma = \{g_1, \dots, g_n\}$, where $n$ is the number of edges of $D$, and the set of infinite sequences $\Lambda$ is the set of Morse codes of generic geodesics on $\H$ with respect to $\D$.

First we prove that if $\D$ has a finite vertex, the set of Morse codes, generated by $\D$, is not a topological Markov chain.

We want to prove that given any $k \in \N$, there can be found a sequence $\lambda_{k+l} = (h_0, h_1, \dots, h_k, h_{k+1}, \dots, h_{k+l})$ of elements of $\Sigma$, for some $l \in \N$, such that there exist generic geodesics $\g^{(i)}$, for $0 \leq i \leq l$, so that Morse code of $\g^{(i)}$ contains $(h_i, h_{1+i}, \dots, h_{k+i})$, $0 \leq i \leq l$, but no geodesic has its Morse code containing $\lambda_{k+l}$ as a subsequence.

Fix some $k \in \N$. Consider a finite vertex $V$. Consequently applying Lemma \ref{lemma:n+1}, we can find a geodesic $\g_0$, passing through $V$ and crossing at least $k$ edges, with the last edge having a finite vertex on the right from $\g_0$. Without loss of generality we may assume that exactly the $k$-th edge that $\g_0$ crosses after passing $V$ has a finite vertex and the vertex is on the right from $\g_0$. Since there are only countably many finite vertices on the plane, we can choose $\g_0$ in such a way, that it doesn't pass through any finite vertex other than $V$.

Denote the endpoints of the first $k$ edges $\g_0$ crosses after passing through $V$, finite or infinite, by $A_1, B_1; \dots; A_k, B_k$ with $A_i$'s being on the left from $\g_0$.

Let $\{\g_\ph\}_{0 \leq \ph < 2\pi}$ be the family of geodesics, passing through $V$, such that the angle between $\g_0$ and $\g_\ph$ in the clockwise direction is $\ph$. It is clear that for all sufficiently small $\ph > 0$ the first $k$ edges $\g_\ph$ crosses after it passes through $V$, are the same as for $\g_0$.

\begin{figure}[hbt]
	\includegraphics[width=\textwidth]{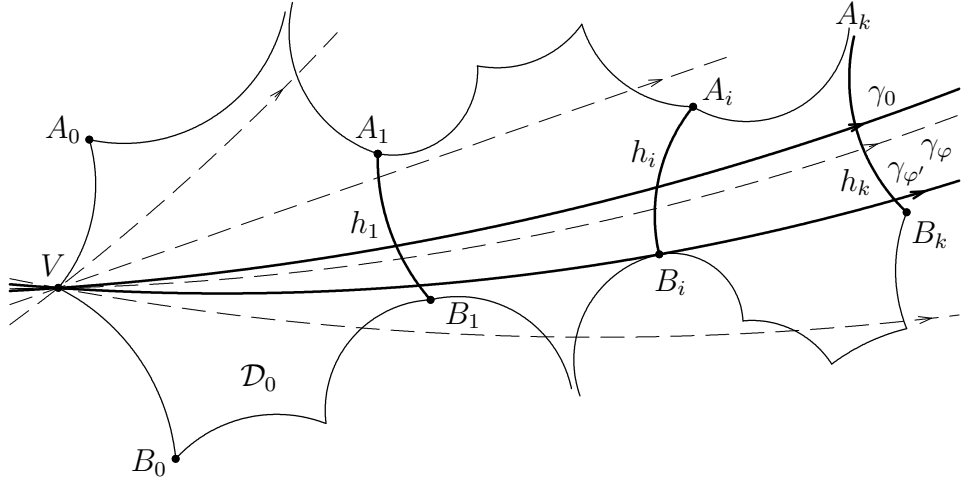}
	\setlength{\unitlength}{\textwidth}
	\begin{picture}(0,0)
		\put(-.46,.38){$A_0$}
		\put(-.465,.235){$V$}
		\put(-.37,.03){$B_0$}
		\put(-.255,.12){$\D_0$}

		\put(-.135,.38){$A_1$}
		\put(-.14,.28){$h_1$}
		\put(-.04,.185){$B_1$}

		\put(.23,.42){$A_i$}
		\put(.155,.355){$h_i$}
		\put(.18,.22){$B_i$}

		\put(.37,.5){$A_k$}
		\put(.375,.325){$h_k$}
		\put(.45,.27){$B_k$}

		\put(.405,.42){$\g_0$}
		\put(.46,.365){$\g_{\ph}$}
		\put(.425,.34){$\g_{\ph'}$}
	\end{picture}
	\caption{The family of geodesics passing through $V$.}
\end{figure}

Let $\ph' = \inf \{\ph > 0$, such that the first $k$ edges, crossed by $\g_\ph$ after it passes through $V$, are different from those of $\g_0\}$. Obviously, $\ph' > 0$. It is also clear that $\g_{\ph'}$ should be a geodesic, intersecting all of $\langle A_1, B_1\rangle$, $\langle A_2, B_2\rangle$, $\dots$, $\langle A_k, B_k]$ within their relative interior or at the right endpoint and passing through at least one of the right endpoints.

The above construction of $\g_{\ph'}$ can be thought of as a clockwise rotation of $\g_0$ around $V$ until it meets the first of $B_i$'s.

Now consider the fundamental domain, which $\g_0$ enters immediately after it passes through $V$, name it $\D_0$. Let the edges of $\D_0$, that end at $V$, have the other endpoints denoted $A_0$ and $B_0$, so that $A_0$ is on the left from $\g_0$.

It can be easily seen, that by another small rotation of $\g_0$ around one of its points other than $V$, one can find a generic geodesic $\g^b$, which intersects $[V,B_0\rangle$, $\langle A_1, B_1\rangle$, $\dots$, $\langle A_k, B_k]$. Let the label at $[V, B_0\rangle$ outside $\D_0$ be $h_0$, and the labels on $\langle A_i, B_i\rangle$ on the side facing $V$ be $h_i$, $1 \leq i \leq k$. Then the sequence $(h_0, h_1, \dots, h_k)$ is a part of Morse code of $\g^b$. Denote this sequence by $\lambda^{(0)}$.

On the other hand, rotating $\g_{\ph'}$ clockwise around a point $w$ behind $\langle A_k, B_k]$, it is possible to find another generic geodesic $\g^a$, which intersects $\langle A_0, V]$, $\langle A_1, B_1\rangle$, $\dots$, $\langle A_k, B_k]$.

\begin{figure}[htb]
	\includegraphics[width=\textwidth]{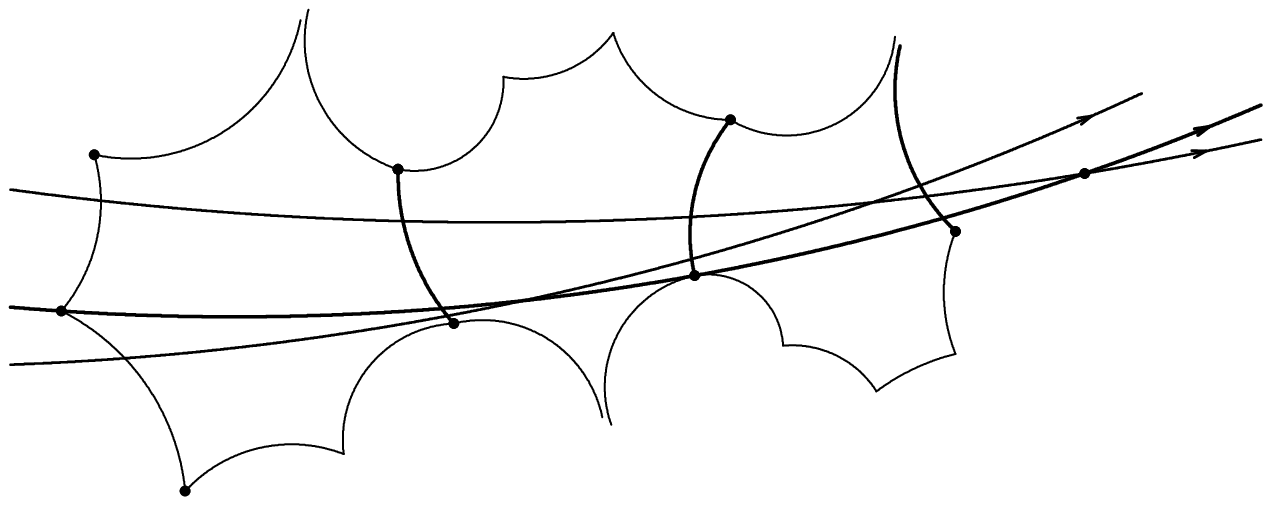}
	\setlength{\unitlength}{\textwidth}
	\begin{picture}(0,0)
		\put(-.455,.325){$A_0$}
		\put(-.49,.195){$V$}
		\put(-.4,.025){$B_0$}
		\put(-.33,.215){$\D_0$}

		\put(-.21,.315){$A_1$}
		\put(-.16,.14){$B_1$}

		\put(.06,.35){$A_i$}
		\put(.035,.18){$B_i$}

		\put(.215,.4){$A_k$}
		\put(.26,.23){$B_k$}

		\put(.35,.27){$w$}

		\put(.315,.34){$\g^b$}
		\put(.43,.35){$\g_{\ph'}$}
		\put(.45,.285){$\g^a$}
	\end{picture}
	\caption{Geodesics $\g_{\ph'}$, $\g^b$, and $\g^a$.}
\end{figure}

Due to Lemma \ref{lemma:bdry}, between $\g^a$ and $\g_{\ph'}$ after they cross there should be an infinite number of finite vertices. We could have chosen $\g^a$ so that one of those finite vertices is close enough to $\g^a$, so we can apply Lemma \ref{lemma:prx}. Let $Z$ be such a finite vertex between the two geodesics, that an edge, ending at $Z$, crosses $\g^a$.

Let the edges, crossed by $\g^a$ after it crosses $\langle A_k, B_k]$ and up to the moment it crosses the edge, ending at $Z$, be labeled by $h_{k+1},\dots,h_{k+l}$, and have endpoints $A_{k+1}$, $B_{k+1}$; $\dots$; $Z$, $B_{k+l}$. This automatically means that sequences $\lambda^{(i)} = (h_i, h_{1+i}$, $\dots$, $h_{k+i})$, $0 < i \leq l$, all are parts of Morse code of $\g^a$.

\begin{figure}[htb]
	\includegraphics[width=\textwidth]{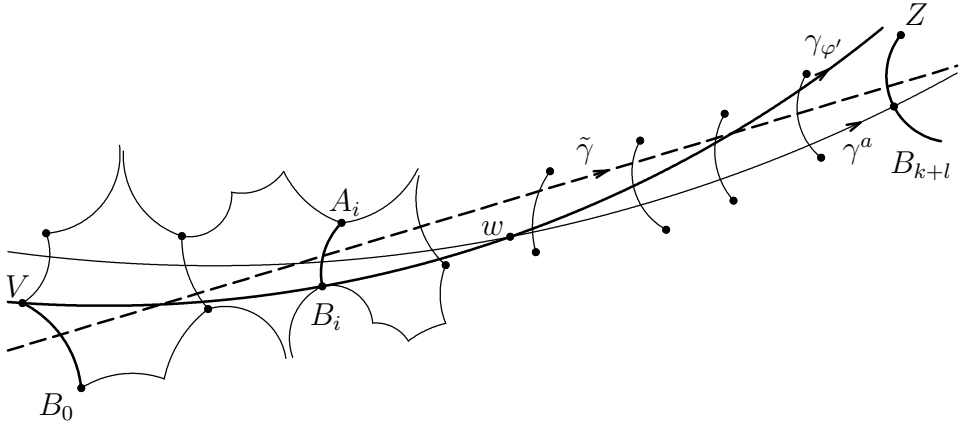}
	\setlength{\unitlength}{\textwidth}
	\begin{picture}(0,0)
		\put(-.5,.14){$V$}
		\put(-.465,.015){$B_0$}

		\put(-.16,.23){$A_i$}
		\put(-.18,.11){$B_i$}

		\put(.0,.205){$w$}
		\put(.1,.285){$\gt$}

		\put(.34,.4){$\g_{\ph'}$}
		\put(.38,.285){$\g^a$}

		\put(.445,.425){$Z$}
		\put(.43,.27){$B_{k+l}$}
	\end{picture}
	\caption{It is impossible for a geodesic to cross $[V, B_0\rangle$, $\langle A_1, B_1\rangle$, $\dots$, $\langle A_k, B_k]$, and $[Z,B_{k+l}\rangle$ all at the same time.}
\end{figure}

Let us show that the sequence $(h_0, h_1, \dots, h_k, h_{k+1}, \dots, h_{k+l})$ cannot be a part of Morse code of any geodesic. Assume this sequence is Morse code of some geodesic. Then there should one such geodesic $\gt$ that crosses $[V, B_0\rangle$, $\langle A_1, B_1\rangle$, $\dots$, $[Z, B_{k+l}\rangle$. But this is impossible, because $[V, B_0\rangle$ and $[Z, B_{k+l}\rangle$ lie on the same side from $\g_{\ph'}$, and $[Z, B_{k+l}\rangle$ does not intersect $\g_{\ph'}$, while at least one of $\langle A_i, B_i\rangle$, $1 \leq i \leq k$, lies on the other side of $\g_{\ph'}$.

So we proved, that $(h_0, h_1, \dots, h_{k+l})$ is not a part of a geodesic's Morse code, while $(h_0, \dots, h_k)$ is a part of Morse code of $\g^b$, and $(h_i, \dots, h_{k+i})$, $0 < i \leq l$, are parts of Morse code of $\g^a$. This ends the proof of necessity.
\bigskip

Now assume that $\D$ does not have finite vertices. Then every edge of $\D$ is a complete geodesic on $\H$.

Recall that set $\Sigma$ consists of the labels put on the edges of $\D$, which are some elements of $\G$. Since any edge is labeled by inverse elements from the opposite sides, and the same label does not appear on different edges inside a single domain, two edges, crossed by a geodesic subsequently, cannot be labeled by inverse elements. Consider set $\Lambda = \{\lambda \in \Sigma^\Z: \lambda_{n+1} \neq \lambda_n^{-1},\ \forall n \in \Z\}$. If we introduce set $\Lambda_1 = \{(g, h) \in \Sigma^2: g \neq h^{-1}\}$, set $\Lambda$ can be written as $\{\lambda \in \Sigma^\Z: (\lambda_n, \lambda_{n+1}) \in \Lambda_1, \forall n \in \Z\}$. It is clear, that Morse code of any geodesic must be a subset of $\Lambda$. In order to prove that Morse coding of geodesics with respect to $\D$ is a topological Markov chain, we will show, that every sequence from $\Lambda$ can be realized as Morse code of some geodesic.

Consider a sequence $\lambda' \in \Lambda$. Note that we can always find a sequence of copies of $\D$, $(\dots, \D_{-1}, \D_0, \D_1, \dots)$, such that $\D_n$ and $\D_{n+1}$ share an edge, which is labeled by $\lambda'_n$ inside $\D_n$. To end the proof, we only need to find a geodesic, which crosses all the edges in the same order. If such a geodesic exists, it cannot cross any other edges in between of the ones just mentioned, since all domains are convex sets, and any two consecutive edges in the sequence belong to a same domain.

Let the edge, shared by $\D_n$ and $D_{n+1}$, have endpoints $\xi_n, \xi_n' \in \Hfty$. It is clear, that $\xi_{n+l}$ and $\xi_{n+l}'$ lie on the same side from geodesic $(\xi_n, \xi_n')$, for any $n \in \Z$ and $l \neq 0 \in \Z$. Consider intervals $[\xi_n, \xi_n']$ on $\Hfty$, for all $n \neq 0$, chosen in such a way, that they do not contain $\xi_0$ and $\xi_0'$. Then $[\xi_n, \xi_n']$, $n \neq 0$, form two nested sequences as $n \to \infty$ and as $n \to -\infty$, one on each side from geodesic $(\xi_0, \xi_0')$. Thus, there should be points $\xi_\infty, \xi_{-\infty} \in \Hfty$, such that $\xi_\infty$ is on the other from $(\xi_0, \xi_0')$ side of $(\xi_n, \xi_n')$, for all $n > 0$, and $\xi_{-\infty}$ is on the other from $(\xi_0, \xi_0')$ side of $(\xi_n, \xi_n')$, for all $n < 0$. Obviously, $\xi_\infty$ and $\xi_{-\infty}$ are two different points. Consider the geodesic $(\xi_{-\infty}, \xi_\infty)$. It has to cross all of $(\xi_n, \xi_n')$ and they can only be crossed in the order of increasing $n$.
\end{proof}

The proof of the sufficiency implies the following corollary.

\begin{cor}
\label{cor:1step}
If Morse coding, given by a Dirichlet domain of a Fuchsian group, produces a topological Markov chain (topological Markov chain), it produces a 1-step topological Markov chain.
\end{cor}

\end{document}